\documentclass[1 [leqno,11pt]{amsart}
\usepackage{amssymb, amsmath,amsmath,latexsym,amssymb,amsfonts,amsbsy, amsthm,mathtools,graphicx,CJKutf8,CJKnumb,CJKulem,color}

\numberwithin{equation}{section}

\allowdisplaybreaks

\let\al=\alpha
\let\b=\beta

\let\f=\frac

\let\om=\omega

\let\Om=\Omega

\let\na=\nabla

\let\pa=\partial
\let\ep=\epsilon

\def\cN{\mathcal{N}}

\def\R{\mathbf R}

\def\eqdef{\buildrel\hbox{\footnotesize def}\over =}

\def\bbT{\mathbb{T}}

\newcommand{\beq}{\begin{equation}}
\newcommand{\eeq}{\end{equation}}
\newcommand{\ben}{\begin{eqnarray}}
\newcommand{\een}{\end{eqnarray}}
\newcommand{\beno}{\begin{eqnarray*}}
\newcommand{\eeno}{\end{eqnarray*}}

\newtheorem{theorem}{Theorem}[section]

\newtheorem{lemma}[theorem]{Lemma}
\newtheorem{proposition}[theorem]{Proposition}

\newtheorem{remark}[theorem]{Remark}
\newtheorem{Theorem}{Theorem}[section]

\newtheorem{Corollary}[Theorem]{Corollary}

\begin{document}

\title{Enhanced dissipation for the 2D Couette flow in critical space}
\author{Nader Masmoudi}
\address{Department of Mathematics, New York University in Abu Dhabi, Saadiyat Island, P.O. Box 129188, Abu Dhabi, United Arab Emirates. Courant Institute of Mathematical Sciences, New York University, 251 Mercer Street, New York, NY 10012, USA,}
\email{masmoudi@cims.nyu.edu}
\author{Weiren Zhao}
\address{Department of Mathematics, New York University in Abu Dhabi, Saadiyat Island, P.O. Box 129188, Abu Dhabi, United Arab Emirates.}
\email{zjzjzwr@126.com, wz19@nyu.edu}
\date{\today}

\maketitle

\begin{abstract}
We consider the 2D incompressible Navier-Stokes equations on $\bbT\times \R$, with initial vorticity that is $\delta$ close in $H^{log}_xL^2_{y}$ to $-1$(the vorticity of the Couette flow $(y,0)$). 
We prove that if $\delta\ll \nu^{1/2}$, where $\nu$ denotes the viscosity, then the solution of the Navier-Stokes equation approaches some shear flow which is also close to Couette flow for time $t\gg \nu^{-1/3}$ by a mixing-enhanced dissipation effect and then converges back to Couette flow when $t\to +\infty$. In particular, we show the nonlinear enhanced dissipation and the inviscid damping results in the almost critical space $H^{log}_xL^2_{y}\subset L^2_{x,y}$. 
\end{abstract}

\section{introduction}
In this paper, we consider the 2D incompressible Navier-Stokes equations on $\bbT\times \R$:
\beq
\label{eq:NS}
\left\{
\begin{array}{l}
\pa_tU+U\cdot\nabla U+\nabla P-\nu\Delta U=0,\\
\na\cdot U=0,\\
U|_{t=0}=U_{in}(x,y).
\end{array}\right.
\eeq
where $U=(U^1,U^2)$ and $P$ denote the velocity and the pressure of the fluid respectively. Let $\Om=\pa_xU^2-\pa_yU^1$ be the vorticity, which satisfies
\beq
\label{eq:vorticity}
\Om_t+U\cdot\nabla \Om-\nu\Delta \Om=0.
\eeq
The Couette flow $(y,0)$ is a steady solution of \eqref{eq:NS} with $\Om=-1$. 

We introduce the perturbation. Let $U=(y,0)+V$ and $\Om=-1+\om$, then $\om$ satisfies 
\ben\label{eq:NS2}
\left\{
\begin{array}{l}
\pa_t\omega+y\pa_x\omega-\nu\Delta\om=-V\cdot\na \om,\\ 
V=\na^{\bot}(-\Delta)^{-1}\om,\\
\om|_{t=0}=\om_{in}(x,y),
\end{array}\right.
\een
and $V$ satisfies
\ben\label{eq:NS3}
\left\{
\begin{array}{l}
\pa_tV+y\pa_xV-\nu\Delta V+\na p=-V\cdot\na V-(V_2,0),\\
\na\cdot V=0,\\
V|_{t=0}=V_{in}(x,y).
\end{array}\right.
\een
The enstrophy conservation law $\|\om(t)\|_{L^2}^2+2\nu\int_0^t\|\na\om(s)\|_{L^2}^2ds=\|\om_{in}\|_{L^2}^2$ implies that the solution of \eqref{eq:NS} remains $\delta$-close in $L^2$ to the Couette flow if the initial vorticity is $\delta$-close in $L^2$ to -1. In this paper, we focus on asymptotic stability of the 2D Couette flow. For the linearized equation 
\ben\label{eq:LNS}
\left\{
\begin{array}{l}
\pa_t\omega+y\pa_x\omega-\nu\Delta\om=0,\\
\om|_{t=0}=\om_{in}(x,y),
\end{array}\right.
\een
it is easy to obtain that 
\ben\label{eq: ED_and_ID}
\|\om_{\neq}\|_{L^2_{x,y}}\leq C\|\om_{in}\|_{L^2_{x,y}}e^{-c\nu t^3}\quad \text{and}\quad
\|V_{\neq}\|_{L^2_{t,x,y}}\leq C\|\om_{in}\|_{L^2_{x,y}},
\een
here we use the notation $f_{\neq}(t,x,y)=f(t,x,y)-\f{1}{|\mathbb{T}|}\int_{\mathbb{T}}f(t,x,y)dx$. 
The first inequality in \eqref{eq: ED_and_ID} is the enhanced dissipation and the second one is the inviscid damping. 

However the nonlinear interaction may affect this linear behavior which leads to the fact that the nonlinear enhanced dissipation and inviscid damping are sensitive to the regularity of the perturbation and/or its smallness. Then an interesting question can be proposed in the following two ways:

1. {\it Given a norm $\|\cdot\|_{X}$($X\subset L^2$), determine a $\beta=\beta(X)$ so that for the initial vorticity $\|\om_{in}\|_{X}\ll \nu^{\beta}$ and for $t>0$,
\ben\label{eq: enha-invis}
\|\om_{\neq}\|_{L^2_{x,y}}\leq C\|\om_{in}\|_{X}e^{-c\nu^{\f13}t}\quad \text{and}\quad
\|V_{\neq}\|_{L^2_{t,x,y}}\leq C\|\om_{in}\|_{X},
\een
or the weak enhanced dissipation type estimate
\ben\label{eq: enha-invis-weak}
\|\om_{\neq}\|_{L^2_{t,x,y}}\leq C\nu^{-\f16}\|\om_{in}\|_{X}
\een
holds for the Navier-Stokes equation \eqref{eq:NS2}.}

2. {\it Given $\beta$, is there an optimal function space $X\subset L^2$ so that if the initial vorticity satisfies $\|\om_{in}\|_{X}\ll \nu^{\beta}$, then \eqref{eq: enha-invis} or \eqref{eq: enha-invis-weak} hold for the Navier-Stokes equation \eqref{eq:NS2}?}

These two problems(find the smallest $\beta$ or find the largest function space $X$) are related to each other, since one can gain regularity in a short time by a standard time-weight argument if the initial perturbation is small enough. 

For $\beta=0$, Bedrossian, Masmoudi and Vicol \cite{BMV2016} showed that if $X$ is taken as Gevery-$m$ with $m<2$, then \eqref{eq: enha-invis-weak} holds. 

For $\beta=\f12$, Bedrossian, Vicol and Wang \cite{BVW2018} proved the nonlinear enhanced dissipation and inviscid damping for the perturbation of initial vorticity in $H^s, s>1$. 

The problem is also related to the stability threshold problem for Couette flow. One may refer to \cite{BGM-1,BGM-2,BGM2017,BMV2016,Braz2004,BVW2018,CLWZ2018,DingLin,WZ2018} for more details. 

Our main goal is to prove that the nonlinear enhanced dissipation and inviscid damping estimates \eqref{eq: enha-invis} hold for the nonlinear equations if the initial vorticity is $\nu^{1/2}$-close to -1 in $H^{log}_xL^2_y\eqdef\big\{f:~\|f\|_{H^{log}_xL^2_y}\eqdef\|\ln(e+|D_x|)f\|_{L^2_{x,y}}<\infty\big\}$. 

Our main result is: 
\begin{theorem}\label{thm: main}
Let $\om$ be a solution of \eqref{eq:NS2} with $\nu<1$. Then there exists $\epsilon_0>0$, such that if $\|V_{in}\|_{L^2_{x,y}}+\|\om_{in}\|_{H^{log}_xL^2_y}\leq \epsilon_0 \nu^{\beta}$ for $\beta\geq 1/2$, then 
\beno
\|\om_{\neq}(t)\|_{H^{log}_xL^2_y}\leq Ce^{-c\nu^{1/3}t}\|\om_{in}\|_{H^{log}_xL^2_y},\quad \|\om_0(t)\|_{L^2_y}\leq C\|\om_{in}\|_{L^2_{x,y}}.
\eeno
where $\om_0(t,y)=\f{1}{|\mathbb{T}|}\int_{\mathbb{T}}\om(t,x,y)dx$ and $\om_{\neq}(t,x,y)=\om(t,x,y)-\om_0(t,y)$. \\
Moreover we have the inviscid damping type estimate,
\begin{align*}
&\int_{0}^{+\infty}\|{V}^2_{\neq}(s)\|_{L^{\infty}_{x,y}}^2ds
+\int_{0}^{+\infty}\||D_x|^{\f12}{V}_{\neq}^2(s)\|_{L^2_xL^{\infty}_y}^2ds
+\int_{0}^{+\infty}\|\pa_x{V}_{\neq}^1(s)\|_{L^2_{x,y}}^2ds
\leq C\|\om_{in}\|_{H^{log}_xL^2_y}^2.
\end{align*} 
The constants $c,\, C$ are independent of $\nu$.
\end{theorem}
By the same argument, one may also get:  
\begin{Corollary}
Let $\om$ be a solution of \eqref{eq:NS2} with $\nu<1$. Then for any $\ep>0$, there exists $\epsilon_0>0$, such that if $\|V_{in}\|_{L^2_{x,y}}+\|\om_{in}\|_{H^{\ep}_xL^2_y}\leq \epsilon_0 \nu^{\beta}$ for $\beta\geq 1/2$, then
\beno
\|\om_{\neq}(t)\|_{H^{\ep}_xL^2_y}\leq Ce^{-c\nu^{1/3}t}\|\om_{in}\|_{H^{\ep}_xL^2_y},\quad \|\om_0(t)\|_{L^2_y}\leq C\|\om_{in}\|_{L^2_{x,y}}.
\eeno
where $\om_0(t,y)=\f{1}{|\mathbb{T}|}\int_{\mathbb{T}}\om(t,x,y)dx$ and $\om_{\neq}(t,x,y)=\om(t,x,y)-\om_0(t,y)$. \\
Moreover we have the inviscid damping type estimate,
\begin{align*}
&\int_{0}^{+\infty}\|{V}^2_{\neq}(s)\|_{L^{\infty}_{x,y}}^2ds
+\int_{0}^{+\infty}\||D_x|^{\f12}{V}_{\neq}^2(s)\|_{L^2_xL^{\infty}_y}^2ds
+\int_{0}^{+\infty}\|\pa_x{V}_{\neq}^1(s)\|_{L^2_{x,y}}^2ds
\leq C\|\om_{in}\|_{H^{\ep}_xL^2_y}^2.
\end{align*}
The constants $c,\, C$ are independent of $\nu$.
\end{Corollary}
By the time weight argument, one can show that there exists $T>0$ independent of $\nu$, such that for $\|\om_{in}\|_{L^2}\leq \ep_0\nu^{\f12}/|\ln\nu|$, $\|\ln(|D|+e)\om(t)\|_{L^2_{x,y}}\leq C\ln((\nu t)^{-1}+e)\|\om_{in}\|_{L^2}$ holds for $t\leq T$, which gives $\|\ln(|D|+e)\om(T)\|_{L^2_{x,y}}\leq C\ln((\nu T)^{-1}+e)\|\om_{in}\|_{L^2}\leq C\ep_0\nu^{\f12}$.  Details can be found in the appendix. The following corollary can be obtained by applying Theorem \ref{thm: main} for $t\geq T$. 

\begin{Corollary}
Let $\om$ be a solution of \eqref{eq:NS2} with $\nu<1$. Then there exists $\epsilon_0>0$, such that if $\|V_{in}\|_{L^2_{x,y}}+\|\om_{in}\|_{L^2_{x,y}}\leq \epsilon_0 \nu^{\f12}|\ln\nu|^{-1}$, then
\beno
\|\om_{\neq}(t)\|_{L^2_{x,y}}\leq Ce^{-c\nu^{1/3}t}\|\om_{in}\|_{L^2_{x,y}},\quad \|\om_0(t)\|_{L^2_y}\leq C\|\om_{in}\|_{L^2_{x,y}}.
\eeno
where $\om_0(t,y)=\f{1}{|\mathbb{T}|}\int_{\mathbb{T}}\om(t,x,y)dx$ and $\om_{\neq}(t,x,y)=\om(t,x,y)-\om_0(t,y)$.\\
Moreover we have the inviscid damping type estimate,
\begin{align*}
&\int_{0}^{+\infty}\|{V}^2_{\neq}(s)\|_{L^{\infty}_{x,y}}^2ds
+\int_{0}^{+\infty}\||D_x|^{\f12}{V}_{\neq}^2(s)\|_{L^2_xL^{\infty}_y}^2ds
+\int_{0}^{+\infty}\|\pa_x{V}_{\neq}^1(s)\|_{L^2_{x,y}}^2ds
\leq C\|\om_{in}\|_{L^2_{x,y}}^2.
\end{align*}
The constants $c,\, C$ are independent of $\nu$.
\end{Corollary}
It also implies that for $\beta>1/2$, the space $X$ can be taken as $L^2$ which is the largest space. 

Let us now outline the main idea in the proof of Theorem \ref{thm: main}. We will show that there is a time $t\sim\nu^{-\f13}$, such that for any $\tau\geq 0$ the energy $E(\tau)$ of the nonzero mode $\om_{\neq}$ satisfies $E(t+\tau)\leq \f12E(\tau)$ and that there exists $C$ independent of $t$ and $\tau$ such that for any $s\in [\tau,t+\tau]$, $E(s)\leq CE(\tau).$

Let us start by some heuristic argument. The main difficulty is to control the nonlinear growth. There are three nonlinear terms $V_0^1\pa_x\om_{\neq}$, $V_{\neq}^2\pa_y\om_0$ and $V_{\neq}\cdot\na\om_{\neq}$. Formally, for the first term, due to the fact $V_0^1(s)$ behaves as $V_0^{1}(\tau)$ for $|\tau-s|\leq \nu^{-\f13}$ and $\pa_x\om_{\neq}(s)$ behaves as $\nu^{-\f12}(s-\tau)^{-\f32}\om_{\neq}(\tau+1)$ for $s\in [\tau+1,\tau+t]$ (due to the enhanced dissipation), the effect of the nonlinear interactions from time $\tau$ to $\tau+t$ cause $\nu^{-\f12}$ growth. For the second term, one can only obtain that $\|\pa_y\om_0(s)\|_{L^2(\tau,\tau+t)L^2_y}\leq C\nu^{-\f12}\|\om\|_{L^2}$ due to that fact that the initial vorticity is in $L^2_y$. Thus the effect of the nonlinear interactions also cause $\nu^{-\f12}$ growth. One can use the same argument for the third term. However, since the Sobolev embedding of $H^1$ in $L^{\infty}$ fails in dimension 2, we need to assume that the initial vorticity has some  $\log$-type regularity in the $x$ direction (see \eqref{eq:V_2L2LinftyLinfty} and \eqref{eq:V_1LinftyL1Linfty} in Lemma \ref{Lem:lin-invdam}). Finally to cancel the $\nu^{-\f12}$ growth, we assume the initial perturbation is $\nu^{\f12}$ small. 
\begin{remark}
The log-type regularity in the $x$ direction is not optimal. Actually by the same argument, one can replace it by $(\ln(e+|D_x|))^{\gamma}$ or $(\ln(e+|D_x|))^{\f12}(\ln\ln(e+|D_x|))^{\gamma}$ with $\gamma>\f12$ and so on. 
\end{remark}

\section{Linear enhanced dissipation and inviscid damping}
We consider the linearized Navier-Stokes around $(y,0)$
\ben\label{eq:LNS}
\left\{
\begin{array}{l}
\pa_t\omega+y\pa_x\omega-\nu\Delta\om=0,\\
\om|_{t=0}=\om_{in}(x,y),
\end{array}\right.
\een
Taking the Fourier transform in the $x$ direction, we get 
\ben\label{eq:LNSF}
\left\{
\begin{array}{l}
\pa_t\widehat{\om}+i\al y\widehat{\om}-\nu(\pa_y^2-\al^2)\widehat{\om}=0,\\
\widehat{\om}|_{t=0}=\widehat{\om}_{in}(\al,y).
\end{array}\right.
\een
Now let us introduce the key lemmas for the linearized system \eqref{eq:LNSF}. The following lemma shows the enhanced dissipation for the linearized system. 
\begin{lemma}\label{Lem:lin-enha}
Suppose $\om$ is a solution of the linearized Navier Stokes equation \eqref{eq:LNS} with initial data satisfying $\int_{\mathbb{T}}\om_{in}(x,y)dx=0$. Then there exist $c$ and $C$ such that for any $t\geq 0$, 
\begin{align}
&\label{eq:om-Linfty}
\|{\om}(t,x,y)\|_{H^{log}_xL^2_y}\leq Ce^{-c\nu^{\f13}t}\|{\om}_{in}(x,y)\|_{H^{log}_xL^2_y},\ \\
&\label{eq:pa_xyomL2L2L2}
\|\na{\om}(t,x,y)\|_{L^2_{t}(H^{log}_xL^2_y)}\leq C\nu^{-\f12}\|{\om}_{in}(x,y)\|_{H^{log}_xL^2_y},\\
&\label{eq:pa_xomL1L2L2}
\|\pa_x{\om}(t,x,y)\|_{L^1_t({H^{log}_xL^2_y})}
\leq C\nu^{-\f12}\|{\om}_{in}(x,y)\|_{H^{log}_xL^2_y},\\
&\label{eq:omL^2LinftyLinfty}
\|\ln(|D_x|+e){\om}(t,x,y)\|_{L^2_tL^{\infty}_{x,y}}\leq C\nu^{-\f12}\|{\om}_{in}(x,y)\|_{H^{log}_xL^2_y}.
\end{align}
\end{lemma}
The next lemma gives the inviscid damping for the linearized system. 
\begin{lemma}\label{Lem:lin-invdam}
Suppose $\om$ is a solution of the linearized Navier Stokes equation \eqref{eq:LNS} with initial data satisfying $\int_{\mathbb{T}}\om_{in}(x,y)dx=0$. Let $\psi$ be the stream function so that $V=(\pa_y\psi,-\pa_x\psi)$ and $-\Delta\psi=\om$, then for any $t\geq 0$,
\begin{align}
&\label{eq:V_2L2LinftyLinfty}
\|\pa_x\psi(t,x,y)\|_{L_t^2L^{\infty}_{x,y}}\leq C\|{\om}_{in}(x,y)\|_{H^{log}_xL^2_y},\\
&\label{eq:V_2L2L2Linfty}
\||D_x|^{1/2}\ln(|D_x|+e)\pa_x{\psi}(t,x,y)\|_{L_t^2L^2_xL^{\infty}_{y}}\leq C\|{\om}_{in}(x,y)\|_{H^{log}_xL^2_y},\\
&\label{eq:V_1L2L2L2}
\|\pa_{y}\pa_x\psi(t,x,y)\|_{L^2_{t}({H^{log}_xL^2_y})}\leq C\|{\om}_{in}(x,y)\|_{H^{log}_xL^2_y}.
\end{align}
Moreover the Sobolev embedding theorem gives
\begin{align}\label{eq:V_1LinftyL1Linfty}
&\|\pa_y{\psi}(t,x,y)\|_{L_t^{\infty}L^{\infty}_{x,y}}\leq C\|{\om}_{in}(x,y)\|_{H^{log}_xL^2_y}.
\end{align}
\end{lemma}
We begin the proof of Lemma \ref{Lem:lin-enha}. 
\begin{proof}
Let $\widetilde{\om}(t,\al,\eta)=\int_{\R}\widehat{\om}(t,\al,y)e^{-i\eta y}dy$ be the Fourier transform of $\widehat{\om}$ in $y$.  
Let $W(t,x,y)=\om(t,x+yt,y)$, then $\widehat{W}(t,\al,y)=\widehat{\om}(t,\al,y)e^{i\al yt}$ and $\widetilde{W}(t,\al,\eta)=\int_{\R}\widehat{\om}(t,\al,y)e^{i\al yt}e^{-i\eta y}dy=\widetilde{\om}(t,\al,\eta-\al t)$. 
It is easy to check that
\beno
\pa_t\widetilde{W}+\nu(\al^2+(\eta-\al t)^2)\widetilde{W}=0,
\eeno
thus we obtain that
\ben\label{eq: id-W}
\widetilde{W}(t,\al,\eta)=e^{-\nu\big(\f13\al^2t^3-\eta\al t^2+\eta^2t+\al^2t\big)}\widetilde{\om}_{in}(\al,\eta),
\een
which gives 
\begin{align*}
|\widetilde{\om}(t,\al,\eta)|
&=e^{-\nu\big(\f13\al^2t^3+\eta\al t^2+\eta^2t+\al^2t\big)}|\widetilde{\om}_{in}(\al,\eta+\al t)|\\
&=e^{-\nu\big(\f{1}{21}\al^2t^3+t(\f{\sqrt{7}}{2\sqrt{2}}\eta+\f{\sqrt{2}}{\sqrt{7}}\al t)^2+\f18\eta^2t+\al^2t\big)}|\widetilde{\om}_{in}(\al,\eta+\al t)|\\
&\leq e^{-\f{1}{21}\al^2\nu t^3-\al^2\nu t-\f18\eta^2\nu t}|\widetilde{\om}_{in}(\al,\eta+\al t)|.
\end{align*}
Thus by using Plancherel's theorem, we get that
\beno
\|\widehat{\om}(t,\al,y)\|_{L^2_y}\leq Ce^{-c\nu\big(\al^2t^3+\al^2t\big)}\|\widehat{\om}_{in}(\al,y)\|_{L^2_y},
\eeno
and
\beno
\|(\pa_y,\al)\widehat{\om}(t,\al,y)\|_{L_t^2L^2_y}\leq C\nu^{-\f12}\|\widehat{\om}_{in}(\al,y)\|_{L^2_y},
\eeno
which gives $\|\ln(|D_x|+e)\na{\om}(t,x,y)\|_{L^2_{t,x,y}}\leq C\nu^{-\f12}\|\ln(|D_x|+e){\om}_{in}(x,y)\|_{L^2_{x,y}}$.

Next we prove \eqref{eq:pa_xomL1L2L2}. We get  
\begin{align*}
&\|\ln(|D_x|+e)\pa_x{\om}(t,x,y)\|_{L^1_tL^2_{x,y}}\\
&\leq C\int_0^T\left(\sum_{\al\neq 0}\|\al\ln(|\al|+e)\widehat{\om}(t,\al,y)\|_{L^2_y}^2\right)^{\f12}dt\\
&\leq C\int_0^1\left(\sum_{\al\neq 0}\||\al|\ln(|\al|+e)\widehat{\om}(t,\al,y)\|_{L^2_y}^2\right)^{\f12}dt\\
&\quad+C\int_1^T\left(\sum_{\al\neq 0}\||\al|\ln(|\al|+e)\widehat{\om}(t,\al,y)\|_{L^2_y}^2\right)^{\f12}dt\\
&\leq  C\left(\int_0^1\sum_{\al\neq 0}\||\al|\ln(|\al|+e)\widehat{\om}(t,\al,y)\|_{L^2_y}^2dt\right)^{\f12}\\
&\quad+\int_1^T\f{C}{t^{3/2}\nu^{1/2}}\left(\sum_{\al\neq 0}\|\ln(|\al|+e)\widehat{\om}_{in}(\al,y)\|_{L^2_y}^2\right)^{\f12}dt\\
&\leq C\nu^{-\f12}\|\ln(|\al|+e)\widehat{\om}_{in}(\al,y)\|_{l^2_{\al}L^2_y}^2.
\end{align*}

At last we prove \eqref{eq:omL^2LinftyLinfty}. Here we will use the Littlewood-Paley theory on $\mathbb{T}\times \R$ which can be found in Section 4.1.1. Let us recall the notation that 
\beno
\bigtriangleup_ju=\int_{\R}\sum_{\al}\widetilde{u}(\al,\eta)\widetilde{\Phi}_j(\al,\eta)e^{i\al x+i\eta y}d\eta={\Phi}_j\ast u. 
\eeno
Recall $W(t,x,y)=\om(t,x+yt,y)$. Then by \eqref{eq:be2D} and \eqref{eq: Schur}, we get that
\begin{align*}
&\|\om(t,x,y)\|_{L^2_tL^{\infty}_{x,y}}\leq \|\om(t,x+yt,y)\|_{L^2_tL^{\infty}_{x,y}}\\
&\leq\Big\|\sum_{j\geq 0}\|\bigtriangleup_jW(t,x,y)\|_{L^{\infty}_{x,y}}\Big\|_{L^2_t}
\leq C\Big\|\sum_{j\geq 0}2^{j}\|\bigtriangleup_jW(t,x,y)\|_{L^{2}_{x,y}}\Big\|_{L^2_t}\\
&\leq C\Big\| \sum_{j\geq 0}2^{j}\|\widetilde{W}(t,\al,\eta)\widetilde{\Phi}_j(\al,\eta)\|_{l^{2}_{\al}L^{2}_{y}}\Big\|_{L^2_t}\\
&=C\int_0^{\infty}\sum_{j'\geq 0} \sum_{j\geq 0}2^{j'}2^{j}\|\widetilde{W}(t,\al,\eta)\widetilde{\Phi}_j(\al,\eta)\|_{l^{2}_{\al}L^{2}_{y}}(s)\|\widetilde{W}(t,\al,\eta)\widetilde{\Phi}_{j'}(\al,\eta)\|_{l^{2}_{\al}L^{2}_{y}}(s)ds\\
&\leq C\left(\nu^{-1}\sum_{j\geq 0}\sum_{j'\geq 0}\f{2^{j}2^{j'}}{2^{2j}+2^{2j'}}\|\bigtriangleup_j\widetilde{\om}_{in}\|_{L^{2}_{x,y}}\|\bigtriangleup_{j'}\widetilde{\om}_{in}\|_{L^{2}_{x,y}}\right)^{\f12}\\
&\leq C\nu^{-1/2}\|{\om}_{in}\|_{L^{2}_{x,y}}.
\end{align*}
The last inequality follows from the fact that the kernel $K(j,j')=\f{2^{j}2^{j'}}{2^{2j}+2^{2j'}}$ satisfies the Schur criterion, 
\beno
\sup_{j'\geq 0}\sum_{j\geq 0}\f{2^{j}2^{j'}}{2^{2j}+2^{2j'}}+\sup_{j\geq 0}\sum_{j'\geq 0}\f{2^{j}2^{j'}}{2^{2j}+2^{2j'}}\leq C. 
\eeno
By the same argument, we get 
\beno
\|\ln(e+|D_x|)\om(t,x,y)\|_{L^2_tL^{\infty}_{x,y}}\leq C\nu^{-1/2}\|\ln(e+|D_x|){\om}_{in}\|_{L^{2}_{x,y}}.
\eeno
Thus we proved the lemma. 
\end{proof}

Next we begin the proof of Lemma \ref{Lem:lin-invdam}.
\begin{proof}
Let us first prove \eqref{eq:V_2L2L2Linfty}. 
By the fact that $\widetilde{\psi}(t,\al,\eta)=(\al^2+\eta^2)\widetilde{\om}(t,\al,\eta)$ and by using \eqref{eq: id-W}, we have
\begin{align*}
|\al\widetilde{\psi}(t,\al,\eta)|\leq C\f{|\al|}{|\eta|^2+\al^2}|\widetilde{\om}_{in}(\al,\eta+\al t)|.
\end{align*}
Thus we get by the Minkowski's integral inequality \eqref{eq:M} that
\begin{align*}
&\||D_x|^{\f12}\ln(|D_x|+e)\pa_x\psi(t,x,y)\|_{L^2_tL^2_xL^{\infty}_y}\\
&\leq \|\al^{\f32}\ln(|\al|+e)\widetilde{\psi}(t,\al,\eta)\|_{l^2_{\al}L_t^2L^{1}_{\eta}}
\leq \|\al^{\f32}\ln(|\al|+e)\widetilde{\psi}(t,\al,\eta)\|_{l^2_{\al}L^{1}_{\eta}L_t^2}\\
&\leq C\left(\sum_{\al\neq 0}\Big(\int_{\R}\f{|\al|^{\f32}\ln(|\al|+e)}{|\eta|^2+\al^2}\Big(\int_0^{T}|\widetilde{\om}_{in}(\al,\eta+\al t)|^2dt\Big)^{\f12}d\eta\Big)^2\right)^{\f12}\\
&\leq C\left(\sum_{\al\neq 0}\Big(\int_{\R}\f{|\al|\ln(|\al|+e)}{|\eta|^2+\al^2}\|\widetilde{\om}_{in}(\al,\eta)\|_{L^2_{\eta}}d\eta\Big)^2\right)^{\f12}\\
&\leq C\left(\sum_{\al\neq 0}\|\ln(|\al|+e)\widetilde{\om}_{in}(\al,\eta)\|_{L^2_{\eta}}^2\right)^{\f12}\leq C\|\ln(|D_x|+e){\om}_{in}(x,y)\|_{L^2_{x,y}},
\end{align*}
which implies \eqref{eq:V_2L2L2Linfty}.

The estimate \eqref{eq:V_2L2LinftyLinfty} follows from the \eqref{eq:V_2L2L2Linfty} and the following Sobolev embedding result,
\beno
\left\|f-\f{1}{|2\pi|}\int_{\mathbb{T}}f(x)dx\right\|_{L^{\infty}(\mathbb{T})}\leq C\||D_x|^{\f12}\ln(|D_x|+e)f\|_{L^2(\mathbb{T})}.
\eeno

Next we prove \eqref{eq:V_1L2L2L2}. We have, 
\begin{align*}
&\||\al|\ln(|\al|+e)\pa_y\widehat{\psi}(t,\al,y)\|_{L^2_{t,\al,y}}\leq \||\al|\ln(|\al|+e)\pa_y\widetilde{\psi}(t,\al,\eta)\|_{L^2_{t,\al,\eta}}\\
&\leq C\left(\sum_{\al\neq 0}\int_{\R}\int_0^T\Big(\f{|\al|\ln(|\al|+e)|\eta|}{\al^2+\eta^2}|\widetilde{\om}_{in}(\al,\eta+\al t)|\Big)^2dtd\eta\right)^{\f12}\\
&\leq C\left(\sum_{\al\neq 0}\int_{\R}\Big(\f{|\al||\eta|}{\al^2+\eta^2}\Big)^2|\al|^{-1}d\eta\|\ln(|\al|+e)\widetilde{\om}_{in}(\al,\eta)\|_{L^2_{\eta}}^2\right)^{\f12}\\
&\leq C\|\ln(|\al|+e)\widetilde{\om}_{in}(\al,\eta)\|_{l^2_{\al}L^2_{\eta}}.
\end{align*}

Finally by the Gagliardo-Nirenberg inequality \eqref{eq: GN}, we have
\begin{align*}
&\|\pa_y\widehat{\psi}(t,\al,y)\|_{L_t^{\infty}l^1_{\al}L^{\infty}_{y}}\\
&\leq C\big\|\|\pa_y\widehat{\psi}(t,\al,y)\|_{L^{2}_{y}}^{\f12}\|\pa_y\widehat{\psi}(t,\al,y)\|_{H^{1}_{y}}^{\f12}\big\|_{L_t^{\infty}l^1_{\al}}\\
&\leq C\big\||\al|^{-\f12}(\ln(|\al|+e))^{-1}\||\al|\ln(|\al|+e)\pa_y\widehat{\psi}(t,\al,y)\|_{L^{2}_{y}}^{\f12}\|\ln(|\al|+e)\pa_y\widehat{\psi}(t,\al,y)\|_{H^{1}_{y}}^{\f12}\big\|_{L_t^{\infty}l^1_{\al}}\\
&\leq C\big\|\||\al|^{-\f12}(\ln(|\al|+e))^{-1}\|_{l^2_{\al}}\|\|\ln(|\al|+e)\pa_y\widehat{\psi}(t,\al,y)\|_{L^{2}_{y}}^{\f12}\|_{l^{4}_{\al}}\|\|\ln(|\al|+e)\pa_y\widehat{\psi}(t,\al,y)\|_{H^{1}_{y}}^{\f12}\|_{l^4_{\al}}\big\|_{L_t^{\infty}}\\
&\leq C\|\ln(|\al|+e)\widehat{\om}_{in}(\al,y)\|_{L^2_{\al,y}},
\end{align*}
which gives the last inequality. 
Thus we proved the lemma. 
\end{proof}


\section{Nonlinear enhanced dissipation and inviscid damping}
In this section, we prove the nonlinear enhanced dissipation and inviscid damping. 

For $t> s$, let $S(t,s)f$ solve 
\beno
\left\{
\begin{array}{l}
\pa_t\om+y\pa_x{\om}-\nu\Delta{\om}=0,\\
\om|_{t=s}=f(x,y),
\end{array}\right.
\eeno
with $\int_{\mathbb{T}}f(x,y)dx=0$. 

We now consider the nonlinear equation,
\ben\label{eq:NSomneq0}
\left\{
\begin{array}{l}
\pa_t{\om}_{\neq}+y\pa_x{\om}_{\neq}-\nu\Delta{\om}_{\neq}
=-\mathcal{N}_1-\mathcal{N}_2-\mathcal{N}_3,\\
\om_{\neq}|_{t=0}=P_{\neq0}\om_{in}(x,y),
\end{array}\right.
\een
with 
\beno
\mathcal{N}_1=({V}^1_{\neq}\pa_x{\om}_{\neq})_{\neq}(t,x,y)+({V}^2_{\neq}\pa_y{\om}_{\neq})_{\neq}(t,x,y),
\eeno
$\mathcal{N}_2={V}_0^1(t,y)\pa_x{\om}_{\neq}(t,x,y)$ and 
$\mathcal{N}_3={V}_{\neq}^2(t,x,y)\pa_y{\om}_0(t,y)$,
where ${\om}_0(t,y)$ satisfies
\ben\label{eq:NSom0}
\left\{
\begin{array}{l}
\pa_t{\om}_{0}-\nu\pa_y^2{\om}_{0}=-({V}^1_{\neq}\pa_x{\om}_{\neq})_{0}(t,y)-({V}^2_{\neq}\pa_y{\om}_{\neq})_0(t,y),\\
\om_{0}|_{t=0}=P_{0}\om_{in}(y),
\end{array}\right.
\een
and ${V}_0^1(t,y)$ satisfies
\ben\label{eq:V_0}
\left\{
\begin{array}{l}
\pa_t{V}_0^1-\nu\pa_y^2 {V}_0^1=-({V}^1_{\neq}\pa_x{V}^1_{\neq})_0(t,y)-({V}^2_{\neq}\pa_y{V}^1_{\neq})_0(t,y),\\
V_0^1|_{t=0}=P_0V_{in}^1(y).
\end{array}\right.
\een
We get by the enstrophy conservation law that 
\beq\label{eq:basic energy}
\|{\om}(t)\|_{L^{2}_{x,y}}^2+2\nu\int_0^t\|\na{\om}(s)\|_{L^2_{x,y}}^2ds=\|{\om}_{in}\|_{L^2_{x,y}}^2,
\eeq
which implies
\ben\label{eq:NAom}
\int_0^t\|\pa_y{\om}_{\neq}(s)\|_{L^2_{x,y}}^2ds+\int_0^t\|\pa_y{\om}_0(s)\|_{L^2_{y}}^2ds\leq \f{1}{2\nu}\|\om(0)\|_{L^2_{x,y}}^2.
\een
We also have
\beno
\|e^{t\nu\pa_y^2}f\|_{L^2}\leq \|f\|_{L^2}, \quad 
\int_s^{\infty}\|\pa_ye^{(t-s)\nu\pa_y^2}f\|_{L^2}^2dt\leq \f{1}{\nu}\|f\|_{L^2}^2,
\eeno
and
\begin{align*}
{V}_0^1(t,y)=&{V}_{in}^1(y)
-\int_0^te^{(t-s)\nu\pa_y^2}\Big(({V}^1_{\neq}\pa_x{V}^1_{\neq})_0(s,y)+({V}^2_{\neq}\pa_y{V}^1_{\neq})_0(s,y)\Big)ds,
\end{align*}
and
\begin{align*}
&\om_{\neq}(t+\tau,\al,y)\\
&=S(t,0)\om_{\neq}(\tau,\al,y)-\int_{0}^tS(t,s)\big(\mathcal{N}_1+\mathcal{N}_2+\mathcal{N}_3\big)(s+\tau)ds.
\end{align*}

The proof of Theorem \ref{thm: main} is based on a bootstrap argument. 

Suppose $\|\ln(e+|D_x|)\om_{in}\|_{L^2_{x,y}}+\|V_{in}\|_{L^2_{x,y}}\leq \epsilon_0\nu^{\b}$ and for any $\tau, t+\tau\in [0,T]$ with $t\geq 0$, the following inequalities hold: 
\begin{itemize}
\item[1. ]Uniform bound of $V_0^1$
\beq
\|{V}^1_0(\tau)\|_{L^2_{y}}\leq 8C_0\epsilon_0\nu^{\b}; \label{btsp:V^1_0}
\eeq
\item[2. ]Enhanced dissipation 
\begin{align}
\|\ln(e+|D_x|){\om}_{\neq}(t+\tau)\|_{L^2_{x,y}}&\leq 8{C}_1e^{-c_1\nu^{\f13}t}\|\ln(e+|D_x|){\om}_{\neq}(\tau)\|_{L^2_{x,y}}\label{btsp:om-point}\\
\left(\int_{\tau}^T\|\ln(e+|D_x|)\na{\om}_{\neq}(s)\|_{L^2_{x,y}}^2ds\right)^{\f12}&\leq 8{C}_2\nu^{-\f12}\|\ln(e+|D_x|){\om}_{\neq}(\tau)\|_{L^2_{x,y}}\label{btsp:pa_xyomL2L2L2}\\
\int_{\tau}^T\|\ln(e+|D_x|)\pa_x{\om}_{\neq}(s)\|_{L^2_{x,y}}ds&\leq 8{C}_3\nu^{-\f12}\|\ln(e+|D_x|){\om}_{\neq}(\tau)\|_{L^2_{x,y}}\label{btsp:pa_xomL1L2L2}\\
\left(\int_{\tau}^T\|\ln(e+|D_x|){\om}_{\neq}(s)\|_{L^{\infty}_{x,y}}^2ds\right)^{\f12}&\leq 8{C}_4\nu^{-\f12}\|\ln(e+|D_x|){\om}_{\neq}(\tau)\|_{L^2_{x,y}}\label{btsp:omL2L1Linfty}
\end{align}
\item[3. ]Inviscid damping
\begin{align}
\left(\int_{\tau}^{T}\|V^2_{\neq}(s)\|_{L^{\infty}_{x,y}}^2dt\right)^{\f12}
&\leq 8{C}_5\|\ln(e+|D_x|){\om}_{\neq}(\tau)\|_{L^2_{x,y}},\label{btsp:V_2L2LinftyLinfty}\\
\left(\int_{\tau}^{T}\||D_x|^{\f12}\ln(e+|D_x|)V_{\neq}^2(s)\|_{L_x^2L^{\infty}_{y}}^2dt\right)^{\f12}
&\leq 8{C}_6\|\ln(e+|D_x|){\om}_{\neq}(\tau)\|_{L^2_{x,y}},\label{btsp:V_2L2L2Linfty}\\
\left(\int_{\tau}^{T}\|\ln(e+|D_x|)\pa_x{V}_{\neq}^1(s)\|_{L^2_{x,y}}^2ds\right)^{\f12}
&\leq 8{C}_7\|\ln(e+|D_x|){\om}_{\neq}(\tau)\|_{L^2_{x,y}};\label{btsp:V_1L2L2L2}
\end{align}
\item[4. ]Uniform bound of $V_{\neq}^1$
\beq
\sup_{s\in[\tau,T)}\|V_{\neq}^1(s)\|_{L^{\infty}_{x,y}}
\leq 8{C}_8\|\ln(e+|D_x|){\om}_{\neq}(\tau)\|_{L^2_{x,y}}.\label{btsp:V_1LinftyL1Linfty}
\eeq
\end{itemize}

The constants $c_1,\ep_0$, and ${C}_k\geq 1$, $k=0,1,2,...,8$, will be determined later. 

By choosing $t=\tau$ and $\tau=0$ in \eqref{btsp:om-point}, we get
\ben\label{lem:tau-0}
\|{\om}_{\neq}(\tau)\|_{L^2_{x,y}}\leq \|\ln(e+|D_x|){\om}_{\neq}(\tau)\|_{L^2_{x,y}}\leq 8C_1\ep_0\nu^{\b}.
\een

\begin{proposition}\label{Prop:btsp}
Let $\beta\geq 1/2$. Assume that $\|\om_{in}\|_{H^{log}_{x}L^2_{y}}+\|V_{in}\|_{L^2_{x,y}}\leq \epsilon_0\nu^{\b}$ and that for some $T>0$, the estimate \eqref{btsp:V^1_0}-\eqref{btsp:V_1LinftyL1Linfty} hold on $[0,T]$. Then there exists $\nu_0$ so that for $\nu<\nu_0$ and $\ep_0$ sufficiently small depending only on $c_1$ and $C_k(k=0,...,8)$ (in particular, independent of $T$), these same estimates hold with all the occurrences of $8$ on the right-hand side replaced by $4$. 
\end{proposition}
This proposition implies Theorem \ref{thm: main} by the standard bootstrap argument. 
Now we begin the proof of Proposition \ref{Prop:btsp}. We need the following lemmas. 

\begin{lemma}\label{lem:V_0}
Under the bootstrap assumptions \eqref{btsp:V^1_0} and \eqref{btsp:om-point}, there is a constant $M_1$ independent of $C_1,c_1$ and $\epsilon_0,\nu$ so that
\beno
\|V_0^1(t)\|_{L^{2}_{y}}
\leq M_1\|V_{in}\|_{L^2_{x,y}}+M_1\|\om_{in}\|_{L^2_{x,y}}\epsilon_0\nu^{\b-1/3}C_1/c_1. 
\eeno
\end{lemma}
\begin{proof}
We have 
\begin{align*}
\|{V}_0^1(t)\|_{L^{2}_{y}}
&\leq \|e^{t\nu\pa_y^2}{V}_{in}^1(0)\|_{L^{2}_{y}}
+\left\|\int_0^te^{(t-s)\nu\pa_y^2}\Big(\big({V}^1_{\neq}\pa_x{V}^1_{\neq}\big)_0+\big({V}^2_{\neq}\pa_y{V}^1_{\neq}\big)_0\Big)ds\right\|_{L^2_y}\\
&\leq C\|{V}_{in}\|_{L^2_{x,y}}
+\left\|({V}_{\neq}\cdot\na{V}^1_{\neq})_0\right\|_{L^1_tL^2_y}.
\end{align*}
By the fact that 
\begin{align*}
\|(V_{\neq}\na V^1_{\neq})_0\|_{L^2_y}\leq \|V_{\neq}\|_{L^2_{x}L^{\infty}_y}\|\na V^1_{\neq}\|_{L^2_{x,y}}
\leq \|\om_{\neq}\|_{L^2_{x,y}}^2,
\end{align*}
and the bootstrap assumption \eqref{btsp:om-point}, we have
\begin{align*}
&\left\|({V}_{\neq}\cdot\na{V}^1_{\neq})_0\right\|_{L^1_tL^2_y}
\leq C\int_0^t\|\om_{\neq}(s)\|_{L^2_{x,y}}^2ds\\ 
&\leq C{C}_1\int_0^t e^{-c_1\nu^{1/3}s}ds\|{\om}_{\neq}(0)\|_{L^2_{x,y}}^2
\leq C\|\om_{in}\|_{L^{2}_{x,y}}\epsilon_0\nu^{\b-1/3}C_1/c_1.
\end{align*}
Here we also used the enstrophy conservation law \eqref{eq:basic energy}. 
This gives the lemma. 
\end{proof}
\begin{lemma}\label{lem: nonlinear-est}
Under the bootstrap assumptions \eqref{btsp:V^1_0}-\eqref{btsp:V_1LinftyL1Linfty}, there is a constant $M_2$ independent of $C_k,\ (k=0,...,8)$ and $\epsilon_0,\nu$ so that for any $t,\tau>0$ and $t+\tau<T$, it holds that
\begin{align*}
&\sum_{k=1}^3\|\ln(e+|D_x|)\mathcal{N}_k(s+\tau)\|_{L^1_s([0,t],L^2_{x,y})}\\
&\leq M_2\ep_0\nu^{\beta-\f12}C_1(C_2C_5+C_6C_2+C_2C_0^{\f12}+C_4C_7+C_3C_8)\|\ln(e+|D_x|){\om}(\tau)\|_{L^2_{x,y}}.
\end{align*}
\end{lemma}
\begin{proof}
Let us fist recall the Littlewood-Paley theory and the Bony's decomposition on $\mathbb{T}$ which can be found in Section 4.1.2.

According to the Bony's decomposition, we divide $\cN_1=V^1_{\neq}\pa_x\om_{\neq}+V^2_{\neq}\pa_y\om_{\neq}$ into four terms 
\beno
\cN_1=T_{\pa_x\om_{\neq}}{V^1_{\neq}}+T^*_{V^1_{\neq}}{\pa_x\om_{\neq}}+T^*_{V^2_{\neq}}{\pa_y\om_{\neq}}+T_{\pa_y\om_{\neq}}{V^2_{\neq}}.
\eeno 
Thus we have
\begin{align*}
\|\ln(e+|D_x|)\cN_1(s+\tau)\|_{L^1_s([0,t],L^2_{x,y})}
&=\int_0^t\|\ln(e+|D_x|)\cN_1(s+\tau)\|_{L^2_{x,y}}ds\\
&\leq C\|\ln(e+|D_x|)T_{\pa_x\om_{\neq}}{V^1_{\neq}}\|_{L^1_s([0,t],L^2_{x,y})}\\
&\quad+C\|\ln(e+|D_x|)T^*_{V^1_{\neq}}{\pa_x\om_{\neq}}\|_{L^1_s([0,t],L^2_{x,y})}\\
&\quad+C\|\ln(e+|D_x|)T^*_{V^2_{\neq}}{\pa_y\om_{\neq}}\|_{L^1_s([0,t],L^2_{x,y})}\\
&\quad+C\|\ln(e+|D_x|)T_{\pa_y\om_{\neq}}{V^2_{\neq}}\|_{L^1_s([0,t],L^2_{x,y})}\\
&=N_{1,1}+N_{1,2}+N_{1,3}+N_{1,4}.
\end{align*}
By the bootstrap assumptions \eqref{btsp:V_1L2L2L2}, \eqref{btsp:omL2L1Linfty} and using \eqref{lem:tau-0} and \eqref{eq: Ber3}, we have
\begin{align*}
N_{1,1}&\leq C\|\ln(e+|D_x|)\pa_x{V}^1_{\neq}(s+\tau)\|_{L^2_s([0,t], L^2_{x,y})}\|{\om}_{\neq}(s+\tau)\|_{L^2_s([0,t],L^{\infty}_{x,y})}\\
&\leq CC_7C_4\nu^{-\f12}\|\ln(e+|D_x|)\om_{\neq}(\tau)\|_{L^2_{x,y}}^2\\
&\leq CC_1C_7C_4\ep_0\nu^{\b-\f12}\|\ln(e+|D_x|)\om_{\neq}(\tau)\|_{L^2_{x,y}}.
\end{align*}
By the bootstrap assumptions \eqref{btsp:V_1LinftyL1Linfty} and \eqref{btsp:pa_xomL1L2L2} and using \eqref{eq: Ber1}, we have
\begin{align*}
N_{1,2}&\leq C\|{V}^1_{\neq}(s+\tau,\b,y)\|_{L^{\infty}_sL^{\infty}_{x,y}}\|\ln(e+|D_x|)\pa_x{\om}_{\neq}(s+\tau)\|_{L^1_s([0,t],L^2_{x,y})}\\
&\leq C\ep_0\nu^{\b-\f12}C_1C_8C_3\|\ln(e+|D_x|)\om_{\neq}(\tau)\|_{L^2_{x,y}}.
\end{align*}
By the bootstrap assumptions \eqref{btsp:V_2L2LinftyLinfty} and \eqref{btsp:pa_xyomL2L2L2} and using \eqref{eq: Ber1}, we have
\begin{align*}
N_{1,3}&\leq C\|\widehat{V}_{\neq}^2(s+\tau)\|_{L^2_s([0,t],L^{\infty}_{x,y})}\|\ln(e+|D_x|)\pa_y{\om}_{\neq}(s+\tau)\|_{L^2_s([0,t],L^2_{x,y})}\\
&\leq C\ep_0\nu^{\b-\f12}C_1C_5C_2\|\ln(e+|D_x|)\om_{\neq}(\tau)\|_{L^2_{x,y}}.
\end{align*}
By the bootstrap assumptions \eqref{btsp:V_2L2L2Linfty} and \eqref{btsp:pa_xyomL2L2L2} and using \eqref{eq: Ber2}, we have
\begin{align*}
N_{1,4}&\leq C\||D_x|^{\f12}\ln(e+|D_x|)V_{\neq}^2(s+\tau)\|_{L^2_s([0,t],L^2_xL^{\infty}_y)}\|\pa_y\om_{\neq}\|_{L^2_s([0,t],L^2_{x,y})}\\
&\leq C\ep_0\nu^{\b-\f12}C_1C_6C_2\|\ln(e+|D_x|)\om_{\neq}(\tau)\|_{L^2_{x,y}}.
\end{align*}
We need the $log$-type regularity only in the estimates of $N_{1,2}$ and $N_{1,3}$, which are due to the fact that $V_{\neq}$ is in lower frequency in $x$, so we need to use $L^{\infty}_{x,y}$ estimate on $V_{\neq}$. For $N_{1,2}$ we use the enhanced dissipation and to treat $N_{1,3}$ we use inviscid damping. 

Thus we have finished the estimate of $\mathcal{N}_1$. \\

Now we deal with $\cN_2$.
By the fact that $\|{\om}_0(t,y)\|_{L^2_y}\leq \|{\om}(t,x,y)\|_{L^2_{x,y}}\leq \|{\om}(\tau,x,y)\|_{L^2_{x,y}}$ for any $\tau<t$ the bootstrap assumptions \eqref{btsp:V^1_0} and \eqref{btsp:pa_xomL1L2L2}, we have
\begin{align*}
&\|\ln(e+|D_x|)\cN_2(s+\tau)\|_{L^1_s([0,t],L^2_{x,y})}\\
&\leq C\int_0^t\|{V}_0^1(s+\tau,y)\|_{L^{\infty}}\|\ln(e+|D_x|)\pa_x\om_{\neq}(s+\tau)\|_{L^2_{x,y}}ds\\
&\leq C\|{V}_0^1(\tau,y)\|_{L^2}^{\f12}\|{\om}_{in}(x,y)\|_{L^2_{x,y}}^{\f12}\|\ln(e+|D_x|)\pa_x\om_{\neq}(s+\tau,\al,\cdot)\|_{L_s^1([0,t],L^2_{x,y})}\\
&\leq CC_0^{\f12}C_2\epsilon_0\nu^{\beta-\f12}\|\ln(e+|D_x|)\om_{\neq}(\tau,\al,y)\|_{L^2_{x,y}}.
\end{align*}

At last we deal with $\cN_3$. By the bootstrap assumption \eqref{btsp:V_2L2L2Linfty} and the fact that 
$$
\|\pa_y{\om}_0(s+\tau,\cdot)\|_{L^2_s([0,t],L^2_y)}\leq \|\pa_y{\om}(s+\tau,x,y)\|_{L^2_s([0,t],L^2_{x,y})}\leq C\nu^{-\f12}\|\widehat{\om}(\tau,x,y)\|_{L^2_{x,y}},
$$ 
we have
\begin{align*}
&\|\ln(e+|D_x|)\cN_3(s+\tau)\|_{L^1_s([0,t],L^2_{x,y})}\\
&\leq C\|\ln(e+|D_x|){V}^2_{\neq}(s+\tau,x,y)\|_{L^2_s([0,t],L^2_xL_y^{\infty})}\|\pa_y{\om}_0(s+\tau,y)\|_{L^2_s([0,t],L^2)}\\
&\leq CC_6\epsilon\nu^{\beta-\f12}\|\ln(e+|D_x|)\om_{\neq}(\tau,\al,y)\|_{L^2_{x,y}}.
\end{align*}
Thus we proved the lemma. 
\end{proof}
Now we are in a position to prove Proposition \ref{Prop:btsp}.
\begin{proof}
Under the bootstrap assumptions \eqref{btsp:V^1_0}-\eqref{btsp:V_1LinftyL1Linfty}, there is a constant $M$ independent of $C_k,\ (k=0,...,8)$ and $\epsilon_0,\nu$ so that for any $t,\tau>0$ and $t+\tau<T$, it holds that
\beq\label{eq:1}
\begin{split}
&\|\ln(e+|D_x|){\om}_{\neq}(t+\tau)\|_{L^2_{x,y}}\\
&\leq Me^{-c\nu^{\f13}t}\|\ln(e+|D_x|){\om}_{\neq}(\tau)\|_{L^2_{x,y}}\\
&\quad+MC_1\big(\ep_0\nu^{\beta-\f12}(C_2C_5+C_2C_6+C_2C_0^{\f12}+C_4C_7+C_3C_8)\big)
\|\ln(e+|D_x|){\om}_{\neq}(\tau)\|_{L^2_{x,y}}\\
&\leq M\big(e^{-c\nu^{\f13}t}+5C_1\ep_0\nu^{\beta-\f12}X^2\big)\|\ln(e+|D_x|){\om}_{\neq}(\tau)\|_{L^2_{x,y}},
\end{split}
\eeq
where $X=\max\{C_0,C_2,C_3,C_4,C_5,C_6,C_7,C_8\}$.

By \eqref{eq:om-Linfty}-\eqref{eq:V_1LinftyL1Linfty} and Lemma \ref{lem: nonlinear-est}, we have 
\begin{align*}
&\nu^{\f12}\left(\int_{\tau}^T\|\ln(e+|D_x|)\na{\om}_{\neq}(s)\|_{L^2_{x,y}}^2ds\right)^{\f12}
+\nu^{\f12}\int_{\tau}^T\|\ln(e+|D_x|)\pa_x{\om}_{\neq}(s)\|_{L^2_{x,y}}ds\\
&\quad+\nu^{\f12}\left(\int_{\tau}^T\|{\om}_{\neq}(s)\|_{L^{\infty}_{x,y}}^2ds\right)^{\f12}
+\left(\int_{\tau}^{T}\|{V}^2_{\neq}(s)\|_{L^{\infty}_{x,y}}^2dt\right)^{\f12}\\
&\quad+\left(\int_{\tau}^{T}\||D_x|^{\f12}\ln(e+|D_x|){V}_{\neq}^2(s)\|_{L^2_xL^{\infty}_y}^2ds\right)^{\f12}
+\left(\int_{\tau}^{T}\|\ln(e+|D_x|)\pa_x{V}_{\neq}^1(s)\|_{L^2_{x,y}}^2ds\right)^{\f12}\\
&\quad+\sup_{s\in[\tau,T]}\|{V}_{\neq}^1(s)\|_{L^{\infty}_{x,y}}\\
&\leq \nu^{\f12}\int_{\tau}^{\infty}\|\ln(e+|D_x|)\na S(t,\tau)\om_{\neq}(\tau)\|_{L^2_{x,y}}^2dt
+\nu^{\f12}\int_{\tau}^{\infty}\|\ln(e+|D_x|)\pa_xS(t,\tau)\om_{\neq}(\tau)\|_{L^2_{x,y}}dt\\
&\quad+\nu^{\f12}\left(\int_{\tau}^{\infty}\|S(t,\tau)\om_{\neq}(\tau)\|_{L^{\infty}_{x,y}}dt\right)^{\f12}
+\left(\int_{\tau}^{\infty}\|\pa_x(-\Delta)^{-1}S(t,\tau)\om_{\neq}(\tau)\|_{L^{\infty}_{x,y}}^2dt\right)^{\f12}\\
&\quad+\left(\int_{\tau}^{\infty}\||D_x|^{\f12}\ln(e+|D_x|)\pa_x(-\Delta)^{-1}S(t,\tau)\om_{\neq}(\tau)\|_{L^2_xL^{\infty}_y}^2dt\right)^{\f12}\\
&\quad+\left(\int_{\tau}^{\infty}\|\ln(e+|D_x|)\pa_x\pa_y(-\Delta)^{-1}S(t,\tau)\om_{\neq}(\tau)\|_{L^2_{x,y}}^2dt\right)^{\f12}\\
&\quad+\sup_{t\in[s,\infty)}\|\pa_y(-\Delta)^{-1}S(t,\tau)\om_{\neq}(\tau)\|_{L^{\infty}_{x,y}}
+\sum_{k=1}^3\|\ln(e+|D_x|)\mathcal{N}_k\|_{L^1_s([0,t],L^2_{x,y})}\\
&\leq M_3\big(1+\ep_0\nu^{\beta-\f12}C_1(C_2C_5+C_2C_6+C_2C_0^{\f12}+C_4C_7+C_3C_8)\big)
\|\ln(e+|D_x|){\om}_{\neq}(\tau)\|_{L^2_{x,y}}\\
&\leq M_3\big(1+5\ep_0\nu^{\beta-\f12}C_1X^2\big)
\|\ln(e+|D_x|){\om}_{\neq}(\tau)\|_{L^2_{x,y}},
\end{align*}
where $X=\max\{C_0,C_2,C_3,C_4,C_5,C_6,C_7,C_8\}$. 

By Lemma \ref{lem:V_0}, we get 
\ben\label{eq:V_0-btsp-proof}
\|V_0^1(t)\|_{L^{2}_{y}}
\leq M_1(1+\epsilon_0\nu^{\b-1/3}C_1/c_1)\ep_0\nu^{\b}. 
\een
Here without loss of generality, we assume $M_1\leq M_3$. 

At last we will determine those constants in the bootstrap assumption. 
The proposition holds if we choose the constants $C_k\, (k=0,1,...,8)$ and $\ep_0$, $c_1$ in the following way.
\begin{align*}
&C_k=\max\{M_3,1\}=X,\quad k=0,2,...,8,\\
&C_1=5\max\{M,1\},\quad c_1=\f{c\ln 2}{\ln 4M},\\
&\ep_0=10^{-2}(\max\{M_3,1\})^{-2}(\max\{M,1\})^{-2}c,
\end{align*}
where $M$ is the constant in \eqref{eq:1}. 

Actually we have 
\beq\label{eq: proof-btsp-3-9}
\begin{split}
&\nu^{\f12}\left(\int_{\tau}^T\|\ln(e+|D_x|)\na{\om}_{\neq}(s)\|_{L^2_{x,y}}^2ds\right)^{\f12}\\
&+\nu^{\f12}\int_{\tau}^T\|\ln(e+|D_x|)\pa_x{\om}_{\neq}(s)\|_{L^2_{x,y}}ds\\
&\quad+\nu^{\f12}\left(\int_{\tau}^T\|{\om}_{\neq}(s)\|_{L^{\infty}_{x,y}}^2ds\right)^{\f12}
+\left(\int_{\tau}^{T}\|{V}^2_{\neq}(s)\|_{L^{\infty}_{x,y}}^2dt\right)^{\f12}\\
&\quad+\left(\int_{\tau}^{T}\||D_x|^{\f12}\ln(e+|D_x|){V}_{\neq}^2(s)\|_{L^2_xL^{\infty}_y}^2ds\right)^{\f12}\\
&+\left(\int_{\tau}^{T}\|\ln(e+|D_x|)\pa_x{V}_{\neq}^1(s)\|_{L^2_{x,y}}^2ds\right)^{\f12}
+\sup_{s\in[\tau,T]}\|{V}_{\neq}^1(s)\|_{L^{\infty}_{x,y}}\\
&\leq M_3\big(1+5\ep_0\nu^{\beta-\f12}C_1X^2\big)
\|\ln(e+|D_x|){\om}_{\neq}(\tau)\|_{L^2_{x,y}}\\
&\leq 4X\|\ln(e+|D_x|){\om}_{\neq}(\tau)\|_{L^2_{x,y}}.
\end{split}
\eeq
Thus \eqref{btsp:pa_xyomL2L2L2}-\eqref{btsp:V_1LinftyL1Linfty} hold with all the occurrences of $8$ on the right-hand side replaced by $4$.

Then we get by \eqref{eq:1} that there exists $t_0=(\ln 4M)(c\nu^{\f13})^{-1}$, so that for any $\tau,\tau+t_0\in[0,T]$,
\ben\label{eq: it-om}
\|\ln(e+|D_x|){\om}_{\neq}(\tau+t_0)\|_{L^2_{x,y}}\leq \f12\|\ln(e+|D_x|){\om}_{\neq}(\tau)\|_{L^2_{x,y}},
\een
and for any $0< s\leq t_0$ and $\tau,\tau+s\in[0,T]$,
\ben\label{eq: it-om2}
\|\ln(e+|D_x|){\om}_{\neq}(\tau+s)\|_{L^2_{x,y}}\leq 2M\|\ln(e+|D_x|){\om}_{\neq}(\tau)\|_{L^2_{x,y}}.
\een
For any $t+\tau,\tau\in [0,T]$ with $t\geq0 $, let $t=nt_0+s$ with $n=[t/t_0]\geq 0$ and $s\in (0,t_0]$. Therefore, by \eqref{eq: it-om}, we get for any $t+\tau,\tau\in [0,T]$ with $t\geq0$,
\begin{align*}
\|\ln(e+|D_x|){\om}_{\neq}(t+\tau)\|_{L^2_{x,y}}
&=\|\ln(e+|D_x|){\om}_{\neq}(nt_0+s+\tau)\|_{L^2_{x,y}}\\
&\leq \f{1}{2}\|\ln(e+|D_x|){\om}_{\neq}((n-1)t_0+s+\tau)\|_{L^2_{x,y}}\\
&\leq \f{1}{2^{[t/t_0]}}\|\ln(e+|D_x|){\om}_{\neq}(s+\tau)\|_{L^2_{x,y}}.
\end{align*}
Then by \eqref{eq: it-om2}, it holds that
\beq\label{eq: proof-bt-om}
\|\ln(e+|D_x|){\om}_{\neq}(t+\tau)\|_{L^2_{x,y}}
\leq 2Me^{-(\ln2)t/t_0+1}\|\ln(e+|D_x|){\om}_{\neq}(\tau)\|_{L^2_{x,y}}.
\eeq
According to the definition of $c_1$, $C_1$, we get for any $t>0$
\beno
2Me^{-(\ln2)t/t_0+1}\leq 4{C_1}e^{-c_1\nu^{\f13}t}.
\eeno
Thus \eqref{eq: proof-bt-om} implies that \eqref{btsp:om-point} holds with the occurrence of $8$ on the right-hand side replaced by $4$. 

At last we have
\beno
M_1+M_1\epsilon_0\nu^{\b-1/3}C_1/c_1\leq 4C_0.
\eeno
Then by \eqref{eq:V_0-btsp-proof}, we proved that \eqref{btsp:V^1_0} holds with the occurrence of $8$ on the right-hand side replaced by $4$. 
Thus we proved the proposition.
\end{proof}

\section{Appendix}
\subsection{Littlewood-Paley theory}
In this subsection, we recall some basic facts about the Littlewood-Paley theory. 
\subsubsection{Littlewood-Paley theory on $\mathbb{T}\times \R$}
Let us first recall some basic facts about the Littlewood-Paley theory on $\mathbb{T}\times \R$. Let $\Phi(x,y)$ and $\Phi_0(x,y)$ be two functions in $C^{\infty}(\mathbb{T}\times\R)$ such that their Fourier transform satisfy $\mathrm{supp}\, \widetilde{\Phi}\subset\left\{\xi=(\al,\eta):\, \f34\leq |\xi|\leq \f83\right\}$, $\mathrm{supp}\, \widetilde{\Phi}_0\subset\left\{\xi=(\al,\eta):\, |\xi|\leq \f43\right\}$ and $\widetilde{\Phi}_0(\xi)+\sum_{j\geq 1}\widetilde{\Phi}_j(\xi)=1$ with $\widetilde{\Phi}_{j}(\xi)=\widetilde{\Phi}(2^{-(j-1)}\xi)$, $j=1,2,...$. 

The Littlewood-Paley operators $\bigtriangleup_j\, (j\geq 0)$ on $\mathbb{T}\times \R$ defined by
\beno
\bigtriangleup_ju=\int_{\R}\sum_{\al}\widetilde{u}(\al,\eta)\widetilde{\Phi}_j(\al,\eta)e^{i\al x+i\eta y}d\eta={\Phi}_j\ast u. 
\eeno
Then the Berstein's inequality gives 
\beq\label{eq:be2D}
\|\bigtriangleup_ju\|_{L^{\infty}_{x,y}}\leq \|{\Phi}_j\|_{L^2_{x,y}}\sum_{|k-j|\leq 2}\|\bigtriangleup_ku\|_{L^2_{x,y}}\leq C2^j\sum_{|k-j|\leq 2}\|\bigtriangleup_ku\|_{L^2_{x,y}}.
\eeq

\subsubsection{Littlewood-Paley theory on $\mathbb{T}$}
Let us recall some basic facts about the Littlewood-Paley theory on $\mathbb{T}$. Let $\phi$ and $\phi_0$ be two functions in $C^{\infty}(\mathbb{T})$ such that $\mathrm{supp}\, \widehat{\phi}\subset\{\f34\leq |\xi|\leq \f83\}$, $\mathrm{supp}\, \widehat{\chi}\subset\{|\xi|\leq \f43\}$ and $\widehat{\chi}(\xi)+\sum_{j\geq 0}\widehat{\phi}(2^{-j}\xi)=1$. 

Then the Littlewood-Paley operators $\bigtriangleup_j, S_j, (j\geq 0)$ on $\mathbb{T}$ defined by
\beno
&&\bigtriangleup_ju=(\phi_j\ast u)(x)=\int_{\mathbb{T}}\phi_j(x-x_1)u(x_1)dx_1,\quad j\geq 0\\
&&S_ju=\sum_{\ell=-1}^{j-1}\bigtriangleup_\ell u=(\chi_j\ast u)(x),\quad \bigtriangleup_{-1}u=(\chi\ast u)(x)
\eeno
Here $\phi_j(x)=2^j\phi(2^jx)$ and $\|\chi_j\|_{L^2}\leq C2^{\f12j}$. 

Furthermore, we have the Bony's decomposition: $T_fg=\sum_{j\geq 1}S_{j-1}f\bigtriangleup_jg$ and $T^*_gf=fg-T_fg$. 

The following Berstein type inequalities will be used.
\ben
\label{eq: Ber1}&&\|\ln(e+|D_x|)T_fg\|_{L^2_x}+\|\ln(e+|D_x|)T^*_fg\|_{L^2_x}\leq C\|f\|_{L^{\infty}_x}\|\ln(e+|D_x|)g\|_{L^2_x},\\
\label{eq: Ber2}&&\|\ln(e+|D_x|)T_fg\|_{L^2_x}\leq C\|f\|_{L^{2}_x}\||D_x|^{\f12}\ln(e+|D_x|)g\|_{L^2_x},\\
\label{eq: Ber3}&&\|\ln(e+|D_x|)T_{\pa_xf}g\|_{L^2_x}\leq C\|f\|_{L^{\infty}_x}\|\ln(e+|D_x|)\pa_xg\|_{L^2_x}.
\een

Here we show the proof of \eqref{eq: Ber2}, \eqref{eq: Ber1} and \eqref{eq: Ber3} can be obtained by the same argument. Indeed, we have
\begin{align*}
&\|\ln(e+|D_x|)T_fg\|_{L^2_x}
=\|\ln(e+|D_x|)(\sum_{j\geq 1}S_{j-1}f\bigtriangleup_jg)\|_{L^2_x}\\
&\leq C\left(\sum_{k\geq -1}\|\langle k\rangle\bigtriangleup_k(\sum_{j\geq 1}S_{j-1}f\bigtriangleup_jg)\|_{L^2_x}^2\right)^{\f12}\\
&\leq C\left(\sum_{k\geq -1}\langle k\rangle\sum_{|j-k|\leq 2}\sup_{j\leq k+2}\|S_{j-1}f\|_{L^{\infty}}^2\|\bigtriangleup_jg\|_{L^2_x}^2\right)^{\f12}\\
&\leq C\left(\sum_{k\geq -1}\langle k\rangle\sum_{|j-k|\leq 2}2^j\sup_{j\leq k+2}\|S_{j-1}f\|_{L^{2}}^2\|\bigtriangleup_jg\|_{L^2_x}^2\right)^{\f12}\\
&\leq C\|f\|_{L^2}\left(\sum_{k\geq -1}\langle k\rangle2^k\|\bigtriangleup_kg\|_{L^2_x}^2\right)^{\f12}
\leq C\|f\|_{L^{2}_x}\||D_x|^{\f12}\ln(e+|D_x|)g\|_{L^2_x}.
\end{align*}
Details of the Littlewood-Paley theory on $\mathbb{T}$ or $\mathbb{T}\times\R$ as well as the Bony's decomposition can be found in \cite{BCD-book, Chemin-book, Danchin-note}.

\subsection{Functional inequalities}
In this subsection, we introduce some basic functional inequalities which are used in the proof. We start with the well-known Gagliardo-Nirenberg on $\R$ inequality (see \cite{Nirenberg}). Suppose $u\in \mathcal{S}(\R)$, then there exists a constant $C$ such that 
\beq\label{eq: GN}
\|u\|_{L^{\infty}(\R)}\leq C\|u\|_{L^{2}(\R)}^{\f12}\|\pa_yu\|_{L^{2}(\R)}^{\f12}. 
\eeq

We also introduce the Minkowski's integral inequality (see \cite{Stein}). Suppose that $(S_1,\mu_1)$ and $(S_2,\mu_2)$ are two $\sigma$-finite measure spaces and $F(x,y):\, S_1\times S_2\to \R$ is measurable. Then it holds for $p>1$ that
\beq\label{eq:M}
\begin{split}
\|F\|_{L^p(d\mu_2, L^1(d\mu_1))}
&\eqdef\left(\int_{S_2}\left|\int_{S_1}F(x,y)d\mu_1(x)\right|^pd\mu_2(y)\right)^{\f1p}\\
&\leq \int_{S_1}\left(\int_{S_2}\left|F(x,y)\right|^pd\mu_2(y)\right)^{\f1p}d\mu_1(x)\eqdef \|F\|_{L^1(d\mu_1, L^p(d\mu_2))}.
\end{split}
\eeq

We end this subsection by introducing the discrete Schur test. Let $K(j,j')$ be the non-negative function defined on $\mathbb{N}^2$ and 
\beno
T(f)(j)=\sum_{j'\in \mathbb{N}}K(j,j')f(j').
\eeno
Then if there exists a constant $C>0$ such that the kernel $K(j,j')$ satisfies 
\beno
\sup_{j\geq 0}\sum_{j'\in \mathbb{N}}K(j,j')\leq C,
\quad \sup_{j'\geq 0}\sum_{j\in \mathbb{N}}K(j,j')\leq C.
\eeno
Then it holds that,
\beq\label{eq: Schur}
\left|\sum_{j\in \mathbb{N}} T(f)(j)g(j)\right|\leq C\|f\|_{l^2}\|g\|_{l^2}
\eeq
\begin{proof}
We only need to prove that $\|T(f)\|_{l^2}\leq C\|f\|_{l^2}$. Using the Cauchy–Schwarz inequality, we have
\begin{align*}
|T(f)(j)|^2=\left|\sum_{j'\in \mathbb{N}}K(j,j')f(j')\right|^2
\leq \left(\sum_{j'\in \mathbb{N}}K(j,j')\right)\left(\sum_{j'\in \mathbb{N}}K(j,j')f(j')^2\right),
\end{align*}
and then by the Fubini's theorem, we get 
\begin{align*}
\|T(f)\|_{l^2}^2
&\leq \sum_{j\in \mathbb{N}}\left(\sum_{j'\in \mathbb{N}}K(j,j')\right)\left(\sum_{j'\in \mathbb{N}}K(j,j')f(j')^2\right)\\
&\leq \left(\sup_{j\in \mathbb{N}}\sum_{j'\in \mathbb{N}}K(j,j')\right)\left(\sup_{j'\in \mathbb{N}}\sum_{j\in \mathbb{N}}K(j,j')\right)\|f\|_{l^2}^2\leq C\|f\|_{l^2}^2.
\end{align*}
Thus we proved \eqref{eq: Schur}. 
\end{proof}

\subsection{Regularization estimate}
In this subsection, we show the local in time estimates and  regularization of the viscosity term. 
\begin{lemma}
Let $\beta=\f12+\f12\ep$ with $\ep>0$. 
Let $\om$ be the solution of \eqref{eq:NS3} with initial data $\om_{in}$ satisfying $\|\om_{in}\|_{L^2}\leq \nu^{\beta}$, then there exist $T>0$ independent $\nu$ such that for any $t\leq T$, 
\beno
\||D|^{\ep}\om(t)\|_{L^2_{x,y}}\leq C(t\nu)^{-\ep/2}\|\om_{in}\|_{L^2}.
\eeno
\end{lemma}
\begin{proof}
Recall that from the linearized equation, we get
\begin{align*}
|\widetilde{\om}(t,\al,\eta)|
\leq e^{-c\al^2\nu t^3-c\al^2\nu t-c\eta^2\nu t}|\widetilde{\om}_{in}(\al,\eta+\al t)|,
\end{align*}
which gives 
\beno
\|(t\nu|D|^2)^{\ep/2}\om(t)\|_{L^2_{x,y}}+\nu^{\f12}\|(t\nu|D|^2)^{\ep/2}\na\om(t)\|_{L^2([0,\infty),L^2_{x,y})}\leq C\|\om_{in}\|_{L^2_{x,y}}.
\eeno
Thus 
\beno
\om(t)=\widetilde{S}(t,0)\om_{in}-\int_0^t\widetilde{S}(t,s)(V\cdot\na \om)(s)ds,
\eeno
with 
\beno
\|(t\nu|D|^2)^{\ep/2}\widetilde{S}(t,s)f\|_{L^2_{x,y}}
\leq C\|f\|_{L^2_{x,y}}.
\eeno
Therefore by using the fact that $\nu^{\f12}\|\na \om\|_{L^2_tL^2}\leq C\|\om_{in}\|_{L^2}$,  we get,
\begin{align*}
&\sup_{t\in[0,T]}\|(t\nu|D|^2)^{\ep/2}\om(t)\|_{L^2_{x,y}}\\
&\leq C\sup_{t\in[0,T]}\|(t\nu|D|^2)^{\ep/2}\widetilde{S}(t,0)\om_{in}\|_{L^2_{x,y}}
+C\sup_{t\in[0,T]}\int_0^t\|V\na\om(s)\|_{L^2}ds\\
&\leq C\|\om_{in}\|_{L^2}+\Big(\int_0^T(s\nu)^{-\ep}ds\Big)^{\f12}\|\na \om(s)\|_{L^2_sL^2}\sup_{s\in[0,T]}\|\om(s)\|_{H^{\ep}}(s\nu)^{\ep/2}\\
&\leq C\|\om_{in}\|_{L^2}\left(1+T^{\f12-\f{\ep}{2}}\nu^{-\f{\ep}{2}}\nu^{-1/2}\sup_{t\in[0,T]}\|(t\nu|D|^2)^{\ep/2}\om(t)\|_{L^2_{x,y}}\right).
\end{align*}
By the assumption $\|\om_{in}\|_{L^2}\leq \nu^{\f{1+\ep}{2}}$, we get that there is $T>0$, so that $CT^{\f12-\f{\ep}{2}}\leq \f12$ and then
\beno
\sup_{t\in[0,T]}\|(t\nu|D|^2)^{\ep/2}\om(t)\|_{L^2_{x,y}}\leq C\|\om_{in}\|_{L^2}.
\eeno
Thus we proved the lemma. 
\end{proof}

\begin{lemma}
Let $\om$ be the solution of \eqref{eq:NS3} with initial data $\om_{in}$ satisfying $\|\om_{in}\|_{L^2}\leq \f{\nu^{1/2}}{|\ln\nu|}$, then there exist $T>0$ independent $\nu$ such that for any $t\leq T$, 
\beno
\|\ln(|D|+e)\om(t)\|_{L^2_{x,y}}\leq C|\ln ((\nu t)^{-1}+e)|\|\om_{in}\|_{L^2}.
\eeno
\end{lemma}
\begin{proof}
Recall that from the linearized equation, we get
\begin{align*}
|\widetilde{\om}(t,\al,\eta)|
\leq e^{-c\al^2\nu t^3-c\al^2\nu t-c\eta^2\nu t}|\widetilde{\om}_{in}(\al,\eta+\al t)|.
\end{align*}
Let $\chi$ be the smooth function support in $|\al|\leq (\nu t)^{-1}$, 
then we get
\begin{align*}
\|\ln(|D|+e)\chi(D)\om(t)\|_{L^2_{x,y}}+\nu^{\f12}\|\ln(|D|+e)\chi(D)\na\om(t)\|_{L^2([0,\infty),L^2_{x,y})}\leq C|\ln\nu|\|\om_{in}\|_{L^2_{x,y}},
\end{align*}
and
\begin{align*}
&\left\|{\ln(|D|+e)}(1-\chi(D))\om(t)\right\|_{L^2_{x,y}}\\
&\quad+\nu^{\f12}\|\ln(|D|+e)(1-\chi(D))\na\om(t)\|_{L^2([0,\infty),L^2_{x,y})}\\
&\leq C{\ln((\nu t)^{-1}+e)}\bigg(\left(\sum_{\al}\left\|\f{\ln(|\al|+e)}{\ln((t\nu)^{-1}+e)}(1-\chi(\al))e^{-\nu\al^2t}\widetilde{\om}_{in}\right\|_{L^2_{\eta}}^2\right)^{\f12}\\
&\quad+\left(\sum_{\al}\left\|\f{\ln(|\al|+e)}{\ln((t\nu)^{-1}+e)}(1-\chi(\al))\nu^{\f12}(|\al|+|\eta|)e^{-\nu\al^2t-\nu\eta^2 t}\widetilde{\om}_{in}\right\|_{L^2_{\eta}L^2_t}^2\right)^{\f12}\bigg)
\end{align*}
by the fact that $\f{\ln(|\al|+e)}{\ln((t\nu)^{-1}+e)}(1-\chi(\al))e^{-\nu\al^2t}\leq C\al e^{-\al}\leq C$ for $\nu t\geq \al^{-1}$, we get 
\beno
\|\ln(|D|+e)\om(t)\|_{L^2_{x,y}}+\nu^{\f12}\|\ln(|D|+e)\na\om(t)\|_{L^2([0,\infty),L^2_{x,y})}\leq C\ln((\nu t)^{-1}+e)\|\om_{in}\|_{L^2_{x,y}},
\eeno
for the solution of the linearized equation.

Thus 
\beno
\om(t)=\widetilde{S}(t,0)\om_{in}-\int_0^t\widetilde{S}(t,s)(V\cdot\na \om)(s)ds,
\eeno
with 
\beno
\left\|\f{\ln(|D|+e)}{\ln((\nu t)^{-1}+e)}\widetilde{S}(t,s)f\right\|_{L^2_{x,y}}
\leq C\|f\|_{L^2_{x,y}}.
\eeno
Therefore by using the fact that $\nu^{\f12}\|\na \om\|_{L^2_tL^2_{x,y}}\leq C\|\om_{in}\|_{L^2_{x,y}}$,  we get,
\begin{align*}
&\sup_{t\in[0,T]}\left\|\f{\ln(|D|+e)}{\ln((\nu t)^{-1}+e)}\om(t)\right\|_{L^2_{x,y}}\\
&\leq C\sup_{t\in[0,T]}\left\|\f{\ln(|D|+e)}{\ln((\nu t)^{-1}+e)}\widetilde{S}(t,0)\om_{in}\right\|_{L^2_{x,y}}
+C\sup_{t\in[0,T]}\int_0^t\|V\na\om(s)\|_{L^2}ds\\
&\leq C\|\om_{in}\|_{L^2}+\|V\|_{L^2L^{\infty}}\|\na\om\|_{L^2L^2}\\
&\leq C\|\om_{in}\|_{L^2}+C\nu^{-\f12}\|(\ln(|D|+e))\om(s)\|_{L^2_tL^2_{x,y}}\|\om_{in}\|_{L^2}\\
&\leq C\|\om_{in}\|_{L^2}+C\nu^{-\f12}T^{\f12}\ln((\nu T)^{-1}+e))\|\om_{in}\|_{L^2}^2
\end{align*}
By the assumption $\|\om_{in}\|_{L^2}\leq \f{\nu^{1/2}}{|\ln\nu|}$, we get that there is $T>0$, so that $CT^{\f12}\f{\ln((\nu T)^{-1}+e))}{\ln(\nu^{-1}+e))}\leq \f12$ and then
\beno
\sup_{t\in[0,T]}\left\|\f{\ln(|D|+e)}{\ln((\nu t)^{-1}+e)}\om(t)\right\|_{L^2_{x,y}}\leq C\|\om_{in}\|_{L^2}.
\eeno
Thus we proved the lemma. 
\end{proof}

\end{document}